\def\benm{\begin{enumerate}}
\def\eenm{\end{enumerate}}

\def\bal{\begin{align}}
\def\eal{\end{align}}

\documentclass{amsart}

\usepackage{amsmath, amsthm, amscd, amsfonts,amssymb, graphicx}

\newtheorem{theorem}{Theorem}[section]

\theoremstyle{definition}
\newtheorem{definition}[theorem]{Definition}

\newtheorem{example}[theorem]{Example}

\newtheorem{proposition}[theorem]{Proposition}
\newtheorem{corollary}[theorem]{Corollary}

\theoremstyle{remark}
\newtheorem{remark}[theorem]{Remark}

\numberwithin{equation}{section}

\begin{document}

\title[Continuous Gabor transform for a class of non-Abelian groups]{Continuous Gabor transform for a class of non-Abelian groups}

\author[Arash Ghaani Farashahi]{Arash Ghaani Farashahi$^{*}$}
\address{$^1$ Department of Pure Mathematics, Faculty of Mathematical sciences, Ferdowsi University of
Mashhad (FUM), P. O. Box 1159, Mashhad 91775, Iran.}
\email{ghaanifarashahi@hotmail.com}
\email{ghaanifarashahi@stu-mail.um.ac.ir}

\curraddr{}


\author{Rajabali Kamyabi-Gol}
\address{$^2$ Department of Pure Mathematics, Faculty of Mathematical sciences, Ferdowsi University of
Mashhad (FUM), P. O. Box 1159, Mashhad 91775, Iran.}
\address{Center of Excellence in Analysis on Algebraic Structures (CEAAS),
Ferdowsi University of Mashhad (FUM), P. O. Box 1159, Mashhad 91775, Iran.}
\email{kamyabi@ferdowsi.um.ac.ir}



\date{}


\thanks{$^*$Corresponding author}
\thanks{E-mail addresses: ghaanifarashahi@hotmail.com (A. Ghaani Farashahi), kamyabi@ferdowsi.um.ac.ir (R. Kamyabi-Gol)}

\begin{abstract}
In this article we define the continuous Gabor transform for second countable, non-abelian, unimodular and type I groups and also we investigate a Plancherel formula and an inversion formula for our definition. As an example we show that how these formulas work for the Heisenberg group and also the matrix group ${SL(2,\mathbb{R})}$.
\end{abstract}

\maketitle

\section{\bf{Introduction}}

Many physical quantities including pressure, sound waves, electro fields, voltages, electronic currents and electromagnetic fields vary with time. These quantities are called signals or waveforms. Signals can be described in a time domain or in a frequency domain by the traditional methods of Fourier transform. The frequency description of signals is known as the frequency analysis or the spectral analysis. It was recognized long times ago that a global Fourier transform of a long time signal has little practical value in analyzing the frequency spectrum of the signal.
From the Fourier transform $\widehat{f}(w)$ of a signal $f(t)$, it is always possible to
determine which frequencies were present in the signal. However, there is absolutely no indication as to when these frequencies existed. So, the Fourier transform cannot provide any information regarding either a time evolution of spectral characteristic or possible localization with respect to the time variable. Some signals such as speech signals or ECG signals require the idea of frequency analysis that is local in time.

In general, the frequency of a signal varies with time, so there is a need for a joint time-frequency representation of a signal in order to describe fully the characteristics of the signal. This requires specific mathematical methods which go beyond the classical Fourier analysis. The first who introduced the joint time-frequency representation of a signal, was Gabor (1946). Gabor transform originates from the work of Dennis Gabor \cite{Gab}, in which he used translations and modulations of the Gaussian signal to represent one dimensional signals.
Also, this transform is eventually called such as; the short time Fourier transform, the Weyl-Heisenberg transform or the windowed Fourier transform.
Usually Gabor analysis is investigated on $\mathbb{R}$ and recently other settings have been looked at (See \cite{FithStro}). Using discrete signals, Gabor theory is done on $\mathbb{Z}$, and numerical implementations require to consider finite periodic signals and consequently Gabor theory on finite cyclic groups. In image processing Gabor theory on $\mathbb{R}^2$ or $\mathbb{R}^n$ and also its discrete version on $\mathbb{Z}^n$ and finite abelian groups is necessary. Computer scientists might even argue that the right setting for Gabor analysis are the $p$-adic groups, because their group operation imitate the computer arithmetic most closely. Since Gabor analysis resets mainly on the structure of translations and modulation, it is possible to extend it to other abelian groups.
For more explanation, we refer the readers to the monograph of ${\rm K.Gr\ddot{o}chenig}$ \cite{Gro} or complete work of H.G.Feichtinger and T.Strohmer \cite{FithStro}. One can find a complete extension of this theory to the set up of locally compact abelain groups.
But many groups in physics such as the Heisenberg group and also many applicable groups in engineering such as Motion groups are non-abelian and so that the standard STFT theory in abelian case fails. On the other hand generalization of the Feichtinger algebra and also Gelfand triple to the set up of non-abelian groups via the short time Fourier transform approach will be useful (see \cite{Feich1} and \cite{Feich2}). These facts persist us to find a generalization of the basic STFT theory for non-abelain groups. We recall that passing through the harmonic analysis of abelian groups to the harmonic analysis of non-abelian groups, many useful results and basic concepts in abelian harmonic analysis collapse, which play important roles in the usual Gabor theory. Thus, the extension of Gabor analysis (STFT theory) for non-abelian groups is not trivial and as we shall discuss, the shape of results in this theory change a lot. But fundamental properties are still valid, with different proofs.
In this article we would like to find an appropriate notation of the continuous Gabor transform on a class of non-abelain groups and also we investigate the generalization of its fundamental properties.

Throughout this paper which contains 4 sections, we assume that $G$ is a second countable, type I and unimodular locally compact group. Section 2 is deduced to fix notations and a brief summery on non-abelian Fourier analysis. In section 3, we define the continuous Gabor transform of a square integrable function $f$ on $G$, with respect to the window function $\psi$, as a measurable fields of operators  defined on $G\times\widehat{G}$ by
$$\mathcal{G}_\psi f(x,\pi):=\int_Gf(y)\overline{\psi(x^{-1}y)}\pi(y)^*dy.$$
Finally, in section 4, we study examples of continuous Gabor transform for the Heisenberg group and the matrix group $SL(2,\mathbb{R})$.

\section{\bf{Preliminaries and notations on non-abelian Fourier analysis}}

Let $\mathcal{H}$ be a separable Hilbert space. An operator $T\in\mathcal{B}(\mathcal{H})$ is called a Hilbert-Schmidt operator if for one, hence for any orthonormal basis $\{e_k\}$ of $\mathcal{H}$ we have $\sum_k\|Te_k\|^2<\infty$. The set of all Hilbert-Schmidt operators on $\mathcal{H}$ denoted by ${\rm HS}(\mathcal{H})$ and for $T\in{\rm HS}(\mathcal{H})$ we define Hilbert-Schmidt norm of $T$ as
$\|T\|_{\rm HS}^2:=\sum_k\|Te_k\|^2.$ It can be checked that ${\rm HS}(\mathcal{H})$ is a self adjoint and two sided ideal in $\mathcal{B}(\mathcal{H})$ and when $\mathcal{H}$ is finite-dimensional we have ${\rm HS}(\mathcal{H}_\pi)=\mathcal{B}(\mathcal{H})$, also
we call an operator $T\in\mathcal{B}(\mathcal{H})$ of trace-class, whenever $\|T\|_{\rm tr}:={\rm tr}[|T|]<\infty$,
where ${\rm tr}[T]:=\sum_{k}\langle Te_k,e_k\rangle$ and $|T|=(TT^*)^{1/2}$. For more details about trace-class and Hilbert-Schmidt-operators, we refer the readers to \cite{Mur}.

Let $(A,\mathcal{M})$ be a measurable space. A family $\{\mathcal{H}_\alpha\}_{\alpha\in A}$ of non zero separable Hilbert
spaces indexed by $A$ will be called a field of Hilbert spaces over $A$. A map $\Phi$ on $A$ such that
$\Phi(\alpha)\in\mathcal{H}_\alpha$ for each $\alpha\in A$ will be called a vector field on $A$. We denote the inner product and norm on $\mathcal{H}_\alpha$ by $\langle.,.\rangle_{\alpha}$ and $\|.\|_\alpha$, respectively.
A measurable field of Hilbert spaces over $A$ is a field of Hilbert spaces $\{\mathcal{H}_\alpha\}_{\alpha\in A}$ together with a countable set $\{e_j\}$ of vector fields such that the functions $\alpha\mapsto\langle e_j(\alpha),e_k(\alpha)\rangle$ are measurable for all $j,k$ and also the linear span of $\{e_j(\alpha)\}$ is dense in $\mathcal{H}_\alpha$ for each $\alpha\in A$. Given a measurable field of Hilbert spaces $(\{\mathcal{H}_{\alpha}\}_{\alpha\in A},\{e_j(\alpha)\})$ on $A$, a vector field $\Phi$ on $A$ will be called measurable if $\langle \Phi(\alpha),e_j(\alpha)\rangle_\alpha$ is measurable function on $A$ for each $j$. The direct integral of the spaces $\{\mathcal{H}_{\alpha}\}_{\alpha\in A}$ with respect to a measure $d\alpha$ on $A$ is denoted by
$\displaystyle\int_A^{\bigoplus}\mathcal{H}_\alpha d\alpha$. This is the space of measurable vectors fields $\Phi$ on $A$ such that we have
$\displaystyle\|\Phi\|^2=\int_A\|\Phi(\alpha)\|_\alpha^2d\alpha<\infty.$
Then it is easily follows that $\displaystyle\int_A^{\bigoplus}\mathcal{H}_\alpha d\alpha$ is a Hilbert space with the inner product
$\displaystyle\langle\Phi,\Psi\rangle=\int_A\langle\Phi(\alpha),\Psi(\alpha)\rangle_\alpha d\alpha.$

Henceforth, when $G$ is a locally compact group and $dx$ is a Haar measure on $G$, $\mathcal{C}_c(G)$ consists of all continuous complex-valued functions on $G$ with compact supports and for each $1\le p<\infty$, let $L^p(G)$ stand for the Banach space of equivalence classes of measurable complex valued functions on $G$ whose $p$-th powers are integrable.
Now, let $\pi$ be a continuous unitary representation of $G$ on the Hilbert space $\mathcal{H}_\pi$, for more details and elementary descriptions about the topological group representations see \cite{FollH} or \cite{HR1}. The representation $\pi$ is called primary, if only scaler multiples of the identity belongs to center of $\mathcal{C}(\pi)$. Note that, primary representations are also known as factor representations.
According to the Schur's lemma, Theorem 3.5 of \cite{FollH}, every irreducible representation is primary.
More generally, if $\pi$ is a direct sum of irreducible representations, $\pi$ is primary if and only if all its irreducible subrepresentations are unitarily equivalent. The group $G$ is said to be type I, if every primary representation of $G$ is a direct sum of copies of some irreducible representation.
Also, assume that the dual space $\widehat{G}$ be the set of all equivalence classes $[\pi]$ of irreducible unitary representations $\pi$ of $G$ and we still use $\pi$ to denote its equivalence class $[\pi]$. Note that, we equipped $\widehat{G}$ with the Fell topology. See \cite{FollH}, for a discussion of this topology on $\widehat{G}$.

There is a measure $d\pi$ on $\widehat{G}$, called the Plancherel measure, uniquely determined once the Haar measure on $G$ is fixed. The family $\{{\rm HS}(\mathcal{H}_\pi)\}_{\pi\in\widehat{G}}$ of Hilbert spaces indexed by $\widehat{G}$ is a field of Hilbert spaces over $\widehat{G}$.
Recall that, ${\rm HS}(\mathcal{H}_\pi)$ is a Hilbert space with the inner product $\langle T,S\rangle_{{\rm HS}(\mathcal{H}_\pi)}={\rm tr}(S^*T)$.
The direct integral of the spaces $\{{\rm HS}(\mathcal{H}_\pi)\}_{\pi\in\widehat{G}}$ with respect to $d\pi$, is denoted by $\displaystyle\int^{\bigoplus}_{\widehat{G}}{\rm HS}(\mathcal{H}_\pi)d\pi$ and for convenience we use the notation $\mathcal{H}^2(\widehat{G})$ for it. If $f\in L^1(G)$, the Fourier transform of $f$ is a measurable field of operators over $\widehat{G}$ given by
\begin{equation}\label{0.0}
\mathcal{F}f(\pi)=\widehat{f}(\pi)=\int_Gf(x)\pi(x)^*dx.
\end{equation}
Let $\mathcal{J}^1(G):=L^1(G)\cap L^2(G)$ and $\mathcal{J}^2(G)$ be the finite linear combinations of convolutions of elements of $\mathcal{J}^1(G)$.
In \cite{Seg1}, Segal proved that, when $G$ is a second countable, non-abelian, unimodular and type I group, there is a measure $d\pi$ on $\widehat{G}$, uniquely determine
once the Haar measure $dx$ on $G$ is fixed, which is called the Plancherel measure and satisfies the following properties;
\begin{enumerate}
\item The Fourier transform $f\mapsto \widehat{f}$ maps $\mathcal{J}^1(G)$ into $\mathcal{H}^2(\widehat{G})$ and it extends to a unitary map from $L^2(G)$ onto $\mathcal{H}^2(\widehat{G})$.
\item Also, each $h\in\mathcal{J}^2(G)$ satisfies the Fourier inversion formula
$\displaystyle h(x)=\int_{\widehat{G}}{\rm tr}[\pi(x)\widehat{h}(\pi)]d\pi.$
\end{enumerate}
Commonly class of type I and unimodular groups are compact groups. When $G$ is a compact group due to Theorem 5.2 of \cite{FollH}, each irreducible representation $(\pi,\mathcal{H}_\pi)$ of $G$ is finite dimensional which implies that $\mathcal{B}(\mathcal{H}_\pi)={\rm HS}(\mathcal{H}_\pi)$ and also every unitary representation of $G$ is a direct sum of irreducible representations. If $\pi$ is any unitary representation of $G$, for each $u,v\in\mathcal{H}_\pi$ the functions $\pi_{u,v}(x)=\langle\pi(x)u,v\rangle$ are called matrix elements of $\pi$. If $\{e_j\}$ is an orthonormal basis for $\mathcal{H}_\pi$, we put $\pi_{ij}(x)=\langle\pi(x)e_j,e_i\rangle$. Notation $\mathcal{E}_\pi$ stands for the linear span of the matrix elements of $\pi$ and $\mathcal{E}$ for the linear span of $\bigcup_{[\pi]\in\widehat{G}}\mathcal{E}_\pi$. The Peter-Weyl Theorem, Theorem 3.12 of \cite{FollH}, guarantee that $\mathcal{E}$ is uniformly dense in $\mathcal{C}(G)$,
$L^2(G)=\bigoplus_{[\pi]\in\widehat{G}}\mathcal{E}_\pi,$
and also $\{d_\pi^{-1/2}\pi_{ij}:i,j=1...d_\pi,[\pi]\in\widehat{G}\}$ is an orthonormal basis for $L^2(G)$.
Thus, according to the Peter-Weyl Theorem, if $f\in L^2(G)$ we have
\begin{equation}\label{3}
f=\sum_{[\pi]\in\widehat{G}}\sum_{i,j=1}^{d_\pi}c_{ij}^\pi(f)\pi_{ij},
\end{equation}
where $c_{i,j}^\pi(f)=d_\pi\langle f,\pi_{ij}\rangle_{L^2(G)}$.
If we choose an orthonormal basis for $\mathcal{H}_\pi$ so that $\pi(x)$ is represented by the matrix $(\pi_{ij}(x))$, then $\widehat{f}(\pi)$ is given by the matrix
$\widehat{f}(\pi)_{ij}=d_{\pi}^{-1}c_{ji}^\pi(f)$ and satisfies
$$\sum_{i,j=1}^{d_\pi}c_{ij}^\pi(f)\pi_{ij}(x)=d_\pi\sum_{i,j=1}^{d_\pi}\widehat{f}(\pi)_{ji}\pi_{ij}(x)=d_\pi{\rm tr}[\widehat{f}(\pi)\pi(x)].$$
So, for $f\in L^2(G)$, (\ref{3}) becomes a Fourier inversion formula,
$f(x)=\sum_{[\pi]\in\widehat{G}}d_\pi{\rm tr}[\widehat{f}(\pi)\pi(x)].$
The Parseval formula
\begin{equation}\label{5}
\|f\|_{L^2(G)}^2=\sum_{[\pi]\in\widehat{G}}\sum_{i,j=1}^{d_\pi}d_\pi^{-1}|c_{ij}(f)|^2
\end{equation}
becomes
$$\|f\|_{L^2(G)}^2=\sum_{[\pi]\in\widehat{G}}d_\pi{\rm tr}[\widehat{f}(\pi)^*\widehat{f}(\pi)].$$

\section{\bf{Continuous Gabor transform}}

Throughout this paper, let $G$ be a second countable, non-abelian, unimodular and type I group.
Let $d\sigma$ be the product of the Haar measure $dx$ on $G$ and the Plancherel measure $d\pi$ on $\widehat{G}$. For each $(x,\pi)\in G\times\widehat{G}$, let $\mathcal{H}_{(x,\pi)}=\pi(x){\rm HS}(\mathcal{H}_\pi)$, where
$\pi(x){\rm HS}(\mathcal{H}_\pi)=\{\pi(x)T:T\in{\rm HS}(\mathcal{H}_\pi)\}.$
It can be checked that $\mathcal{H}_{(x,\pi)}$ is a Hilbert space with respect to the inner product $\langle\pi(x)T,\pi(x)S\rangle_{\mathcal{H}_{(x,\pi)}}={\rm tr}(S^*T).$
The family $\{\mathcal{H}_{(x,\pi)}\}_{(x,\pi)\in G\times\widehat{G}}$ of Hilbert spaces indexed by $G\times\widehat{G}$ is a field of Hilbert spaces over $G\times\widehat{G}$. The direct integral of the spaces $\{\mathcal{H}_{(x,\pi)}\}_{(x,\pi)\in G\times\widehat{G}}$ with respect to $\sigma$, is denoted by $\mathcal{H}^2(G\times\widehat{G})$, that is the space of all measurable vector fields $F$ on $G\times\widehat{G}$ such that
$$\|F\|_{\mathcal{H}^2(G\times\widehat{G})}^2=\int_{G\times\widehat{G}}\|F(x,\pi)\|_{(x,\pi)}^2d\sigma(x,\pi)<\infty.$$
It can be checked that  $\mathcal{H}^2(G\times\widehat{G})$ becomes a Hilbert space, with the inner product
$$\langle F,K\rangle_{\mathcal{H}^2(G\times\widehat{G})}=\int_{G\times\widehat{G}}{\rm tr}[K(x,\pi)^*F(x,\pi))]d\sigma(x,\pi).$$
Let $L^2(G,\mathcal{B}(\mathcal{H}_\pi))$ be the Banach space of all measurable functions $\phi:G\to\mathcal{B}(\mathcal{H}_\pi)$ with $$\|\phi\|_{L^2(G,\mathcal{B}(\mathcal{H}_\pi))}^2=\int_G\|\phi(x)\|^2dx<\infty,$$
where for each $x\in G$ by $\|\phi(x)\|$ we mean the operator norm of $\phi(x)$. More explanations about spaces related to functions with values in a Banach space can be found in \cite{50}.
For each $\psi\in L^2(G)$ and $\pi\in\widehat{G}$, the modulation of $\psi$ with respect to $\pi$, is an operator valued mapping defined almost every where on $G$ by
$\psi_\pi(x)={\psi}(x)\pi(x).$
The linear transformation $M_\pi:L^2(G)\to L^2(G,\mathcal{B}(\mathcal{H}_\pi))$ defined by $\psi\mapsto \psi_\pi$ is called the $\pi$-modulation operator. Clearly, $M_\pi$ is an isometry, because for each $\psi\in L^2(G)$ we have
\begin{align*}
\|M_\pi\psi\|^2_{L^2(G,\mathcal{B}(\mathcal{H}_\pi))}&=\int_G\|M_\pi\psi(x)\|^2 dx
\\&=\int_G\|\psi(x)\pi(x)\|^2dx=\|\psi\|_{L^2(G)}^2
\end{align*}
Our definition for modulation coincides with the usual definition of modulation as multiplication by a character in the abelian group case.
 Indeed, when $G$ is abelian, each irreducible representation of $G$ is one-dimensional and the corresponding representation space $\mathcal{H}_\pi$ is isomorphic to $\mathbb{C}$.
For $f\in L^2(G)$ and $\phi\in L^2(G,\mathcal{B}(\mathcal{H}_\pi))$, let $\langle f,\phi\rangle_\pi$ be the bounded operator on $\mathcal{H}_\pi$ defined by
$$\displaystyle\langle f,\phi\rangle_\pi=\int_G f(y){\phi(y)}^*dy.$$
We interpret this operator valued integral in the weak sense. That is, for any $z\in\mathcal{H}_\pi$ we define $\langle f,\phi\rangle_\pi z$ by specifying its inner product with an arbitrary $v\in\mathcal{H}_\pi$ via
\begin{equation}\label{f2}
\left\langle\langle f,\phi\rangle_\pi z,v\right\rangle=\int_Gf(y)\langle\phi(y)^*z,v\rangle dy.
\end{equation}
Since the map $y\mapsto\langle\phi(y)^*z,v\rangle$ belongs to $L^2(G)$, the right hand side integral (\ref{f2}) is the ordinary integral of a function in $L^1(G)$. It is not difficult to see that $|\langle\langle f,\phi\rangle_\pi z,v\rangle|\le\|z\|\|v\|\|f\|_{L^2(G)}\|\phi\|_{L^2(G,\mathcal{B}(\mathcal{H}_\pi))}$, therefore $\langle f,\phi\rangle_\pi$ defines a bounded linear operator on $\mathcal{H}_\pi$, whose norm salsifies $\|\langle f,\phi\rangle_\pi\|\le\|f\|_{L^2(G)}\|\phi\|_{L^2(G,\mathcal{B}(\mathcal{H}_\pi))}$.
We can consider $$\langle.,.\rangle_\pi:L^2(G)\times L^2(G,\mathcal{B}(\mathcal{H}_\pi))\to \mathcal{B}(\mathcal{H}_\pi)$$ as a separately continuous sesqulinear map with values in $\mathcal{B}(\mathcal{H}_\pi)$. When $G$ is an abelian group, for each $\pi\in\widehat{G}$, this sesqulinear map coincides with the usual inner product of $L^2(G)$.
\begin{definition}
Let $\psi$ be a window function (a fixed nonzero function in $L^2(G)$) and $f\in \mathcal{C}_c(G)$. We define the continuous Gabor transform of $f$ with respect to the window function $\psi$ as a measurable fields of operators on $G\times\widehat{G}$ by
\begin{equation}\label{f3}
\mathcal{G}_\psi f(x,\pi):=\int_Gf(y)\overline{\psi(x^{-1}y)}\pi(y)^*dy.
\end{equation}
Before studying the basic properties of our extension for the continuous Gabor transform we verify ambiguous points of this definition.
First, note that we consider the operator-valued integral (\ref{f3}) in the weak sense. In other words, for each $(x,\pi)\in G\times\widehat{G}$ and $\zeta,\xi\in\mathcal{H}_{\pi}$ we have
$$\langle\mathcal{G}_\psi f(x,\pi)\zeta,\xi\rangle=\int_Gf(y)\overline{\psi(x^{-1}y)}\langle\pi(y)^*\zeta,\xi\rangle dy.$$
Since the map $y\mapsto \langle\pi(y)^*\zeta,\xi\rangle$ is a bounded continuous function on $G$, the right hand side integral is the ordinary integral
of a function in $L^1(G)$. It is clear that $|\langle\mathcal{G}_\psi f(x,\pi)\zeta,\xi\rangle|\le\|\zeta\|\|\xi\|\|f\|_{L^2(G)}\|\psi\|_{L^2(G)}$, so $\mathcal{G}_\psi f(x,\pi)$ is indeed a bounded linear operator on $\mathcal{H}_{\pi}$ such that $\|\mathcal{G}_\psi f(x,\pi)\|\le\|f\|_{L^2(G)}\|\psi\|_{L^2(G)}$.

Second, since $f\in\mathcal{C}_c(G)$ and $\psi\in L^2(G)$ we have $f.L_x\psi\in \mathcal{J}^1(G)$ for each $x\in G$. Plancherel theorem, Theorem 7.44 of \cite{FollH}, guarantee that $\widehat{f.L_{x}\psi}(\pi)$ is a Hilbert-Schmidt operator for almost every where $\pi\in\widehat{G}$. Thus, for $\sigma$-almost every $(x,\pi)$ in $G\times\widehat{G}$ we have $\mathcal{G}_\psi f(x,\pi)\in\mathcal{H}_{(x,\pi)}$. Indeed,
\begin{align*}
\mathcal{G}_\psi f(x,\pi)&=\int_Gf(y)\overline{\psi(x^{-1}y)}\pi(y)^*dy
\\&=\int_Gf(y)\overline{\psi(x^{-1}y)}\pi(x)\pi(yx)^*dy
\\&=\pi(x)\left(\int_Gf(y)\overline{\psi(x^{-1}y)}\pi(yx)^*dy\right)
\\&=\pi(x)\left(\int_Gf(yx^{-1})\overline{\psi(x^{-1}yx^{-1})}\pi(y)^*dy\right)
=\pi(x)\mathcal{F}\left(R_{x^{-1}}(f.L_x\overline{\psi})\right)(\pi).
\end{align*}
\end{definition}
In the next proposition we state some worthwhile properties of our definition. First, let us recall that when $G$ is unimodular and $1\le p<\infty$, the involution for $g\in L^p(G)$ is $\widetilde{g}(x)=\overline{g(x^{-1})}$.
\begin{proposition}\label{14}
{\it Let $\psi$ be a window function and $f\in \mathcal{C}_c(G)$. Then, for each $(x,\pi)\in G\times \widehat{G}$ we have
\begin{enumerate}
\item $\mathcal{G}_\psi f(x,\pi)=\widehat{\mathcal{L}_x^\psi(f)}(\pi)$, where $\mathcal{L}_x^\psi(f)$ in $L^1(G)$ is defined for a.e. $y$ in $G$ by $f(y)\overline{\psi(x^{-1}y)}$.
\item $\mathcal{G}_\psi f(x,\pi)^*=\mathcal{F}\left({\widetilde{{\mathcal{L}_x^{\psi}(f)}}}\right)(\pi)$.
\item $\mathcal{G}_\psi f(x,\pi)=\langle f,M_\pi(L_x\psi)\rangle_\pi$.
\end{enumerate}
We call $(1)$ the Fourier representation form of the continuous Gabor transform.
}\end{proposition}
\begin{proof}
(1) and (3) follow immediately from the definitions. Note that, H${\rm\ddot{o}}$lder's inequality guarantee that $\mathcal{L}_x^\psi(f)$ belongs to $L^1(G)$. Let $(x,\pi)\in G\times\widehat{G}$ and $\zeta\in\mathcal{H}_{\pi}$. For (2), using unimodularity of $G$ and the identity $\widetilde{\mathcal{L}_x^\psi(f)}=\widetilde{f}.\overline{\widetilde{L_x\psi}}$, we get
\begin{align*}
\langle\mathcal{G}_\psi f(x,\pi)^*\zeta,\zeta\rangle&=\langle\zeta,\mathcal{G}_\psi f(x,\pi)\zeta\rangle
\\&=\int_G\langle\zeta,f(y)\overline{\psi(x^{-1}y)}\pi(y)^*\zeta\rangle dy
\\&=\int_G\langle\overline{f(y)}{L_x\psi(y)}\pi(y)\zeta,\zeta\rangle dy
\\&=\int_G\langle\widetilde{f}(y)\overline{\widetilde{L_x\psi}}(y)\pi(y)^*\zeta,\zeta\rangle dy
=\left\langle\mathcal{F}\left({\widetilde{{\mathcal{L}_x^{\psi}(f)}}}\right)(\pi)\zeta,\zeta\right\rangle.
\end{align*}
\end{proof}
The next theorem shows that the continuous Gabor transform preserves the energy of the signal. More preicisly, we prove that for each window function $\psi$, the linear operator $\mathcal{G}_\psi:\mathcal{C}_c(G)\to\mathcal{H}^2(G\times\widehat{G})$ given by $f\mapsto\mathcal{G}_\psi f$ is a multiple of an isometry.
\begin{theorem}\label{15}
Let $\psi$ be a window function. Then, for each $f\in \mathcal{C}_c(G)$ we have
$$\|\mathcal{G}_\psi f\|_{\mathcal{H}^2(G\times\widehat{G})}=\|f\|_{L^2(G)}\|\psi\|_{L^2(G)}.$$
\end{theorem}
\begin{proof}
Using Proposition \ref{14}, Theorem 2.1 of \cite{Lip}, Fubini's theorem and also unimodularity of $G$ we have
\begin{align*}
\|\mathcal{G}_\psi f\|_{\mathcal{H}^2(G\times\widehat{G})}^2&=
\int_{G\times \widehat{G}}\|\mathcal{G}_\psi f(x,\pi)\|_{(x,\pi)}^2d\sigma(x,\pi)
\\&=\int_{G\times \widehat{G}}\mathrm{tr}[\mathcal{G}_\psi f(x,\pi)^*\mathcal{G}_\psi f(x,\pi)]d\sigma(x,\pi)
\\&=\int_G\int_{\widehat{G}}\mathrm{tr}[\mathcal{G}_\psi f(x,\pi)^*\mathcal{G}_\psi f(x,\pi)]d\pi dx
\\&=\int_G\int_{\widehat{G}}\mathrm{tr}[{\widehat{{\widetilde{\mathcal{L}_x^{\psi}(f)}}}(\pi)}
\widehat{\mathcal{L}_x^\psi(f)}(\pi)]d\pi dx
\\&=\int_G\int_{\widehat{G}}\mathrm{tr}[{\widehat{{\mathcal{L}_x^{\psi}(f)}}(\pi)}^*
\widehat{\mathcal{L}_x^\psi(f)}(\pi)]d\pi dx
\\&=\int_G\int_{G}\overline{\mathcal{L}_x^{\psi}(f)(y)}
{\mathcal{L}_x^\psi(f)}(y)dy dx
=\|f\|_{L^2(G)}^2\|\psi\|_{L^2(G)}^2.
\end{align*}
\end{proof}

According to Theorem \ref{15}, the continuous Gabor transform $\mathcal{G}_\psi:\mathcal{C}_c(G)\to\mathcal{H}^2(G\times\widehat{G})$ defined by $f\mapsto\mathcal{G}_\psi f$ is a multiple an isometry. So, we can extend $\mathcal{G}_\psi$ uniquely to a bounded linear operator from $L^2(G)$ into a closed subspace of $\mathcal{H}^2(G\times\widehat{G})$ which we still use the notation $\mathcal{G}_\psi$ for this extension and this extension for each $f\in L^2(G)$ satisfies
$$\|\mathcal{G}_\psi f\|_{\mathcal{H}^2(G\times\widehat{G})}=\|f\|_{L^2(G)}\|\psi\|_{L^2(G)}.$$
We call $\mathcal{G}_\psi f$ the continuous Gabor transform of $f\in L^2(G)$ with respect to the window function $\psi$, which can be considered as the sesquilinear map $(f,\psi)\mapsto \mathcal{G}_\psi f$ on $L^2(G)\times L^2(G)$ into $\mathcal{H}^2(G\times\widehat{G})$.

We can conclude the following orthogonality relation for the continuous Gabor transform.
\begin{corollary}\label{16}{\it Let $\psi,\varphi$ be two window functions. The continuous Gabor transform satisfies the orthogonality relation;
$\langle\mathcal{G}_\psi f,\mathcal{G}_\varphi g\rangle_{\mathcal{H}^2(G\times\widehat{G})}=\langle\varphi,\psi\rangle_{L^2(G)}\langle f,g\rangle_{L^2(G)}$, for each $\ f,g\in L^2(G).$
Moreover, the normalized Gabor transform $\|\psi\|_{L^2(G)}^{-1}\mathcal{G}_\psi$ is an isometry from $L^2(G)$ onto a closed subspace of
$\mathcal{H}^2(G\times\widehat{G})$.
}\end{corollary}
\begin{proof} By using the polarization identity, the result follows.
\end{proof}


\begin{remark}
When $G$ is a compact group automatically $G$ is unimodular and also type I and so that we can use the definition (\ref{f3}) of $\mathcal{G}_\psi f$ for all $f,\psi\in L^2(G)$, which first we define the continuous Gabor transform of $f\in\mathcal{C}_c(G)$ and then we extend it for elements of $L^2(G)$. But independent of (\ref{f3}) since $G$ is a compact group we can consider for a window function $\psi$ in $L^2(G)$ and each $f\in L^2(G)$ we can define the continuous Gabor transform of $f$ with respect to the window function $\psi$ as a Hilbert-Schmidt operator valued function on $G\times\widehat{G}$ by
\begin{equation}
\mathcal{G}_\psi f(x,\pi)=\int_Gf(y)\overline{\psi(x^{-1}y)}\pi(y)^*dy.
\end{equation}
In this case by compactness of $G$, it is guaranteed that for each $(x,\pi)\in G\times\widehat{G}$ we have $\mathcal{G}_\psi f(x,\omega)\in{\rm HS}(\mathcal{H}_\pi)$. In the sequel corollary we state the Plancheral formula for the case $G$ is compact. It should be noted that although it is a corollary of the preceding Theorem but it can be proved separately.
\end{remark}

\begin{corollary}
{\it Let $G$ be a compact group and $\psi$ be a window function. Then, for each $f\in L^2(G)$ we have
$$\int_G\sum_{[\pi]\in\widehat{G}}\|\mathcal{G}_\psi f(x,\pi)\|_{\rm HS}^2dx=\|f\|_{L^2(G)}^2\|\psi\|_{L^2(G)}^2.$$
}\end{corollary}

To prove an inversion formula, we need some notations.
Let $\psi$ be a window function and $K\in\mathcal{H}^2(G\times\widehat{G})$. The conjugate linear functional
$$\ell_\psi^K(g)=\int_{G\times\widehat{G}}{\rm tr}[K(y,\pi)\mathcal{G}_\psi g(y,\pi)^*]d\sigma(y,\pi),$$
is a bounded functional on $L^2(G)$. Indeed, using Cauchy-Schwartz inequality and Theorem \ref{15} we have
\begin{align*}
|\ell_\psi^K(g)|&=\left|\int_{G\times\widehat{G}}{\rm tr}[K(y,\pi)\mathcal{G}_\psi g(y,\pi)^*]d\sigma(y,\pi)\right|
\\&\le\int_{G\times\widehat{G}}\left|{\rm tr}[K(y,\pi)\mathcal{G}_\psi g(y,\pi)^*]\right|d\sigma(y,\pi)
\\&\le\|K\|_{\mathcal{H}^2(G\times\widehat{G})}\|\mathcal{G}_\psi g\|_{\mathcal{H}^2(G\times\widehat{G})}
=\|K\|_{\mathcal{H}^2(G\times\widehat{G})}\|\psi\|_{L^2(G)}\|g\|_{L^2(G)}.
\end{align*}
This shows that $\ell_\psi^K$ defines a unique element in $L^2(G)$, which we use the notation
$$\int_{G\times\widehat{G}}{\rm tr}[K(y,\pi)M_\pi(L_y\psi)]d\sigma(y,\pi),$$
for this element of $L^2(G)$. According to this notation, for each $g\in L^2(G)$ we have
$$\left\langle\int_{G\times\widehat{G}}{\rm tr}[K(y,\pi)M_\pi(L_y\psi)]d\sigma(y,\pi),g\right\rangle_{L^2(G)}=\int_{G\times\widehat{G}}{\rm tr}[K(y,\pi)\mathcal{G}_\psi g(y,\pi)^*]d\sigma(y,\pi).$$

In the next theorem we prove an inversion formula.
\begin{theorem}\label{17}
{\it Let $\psi,\varphi$ be two window functions such that $\langle\varphi,\psi\rangle_{L^2(G)}\not=0$. Then, for each $f\in L^2(G)$ we have
$$f=\langle\varphi,\psi\rangle^{-1}_{L^2(G)}\int_{G\times\widehat{G}}{\rm tr}[\mathcal{G}_\psi f(y,\pi)M_\pi(L_y\varphi)]d\sigma(y,\pi).$$
}\end{theorem}
\begin{proof}
By Theorem \ref{15}, we have $\mathcal{G}_\psi f\in\mathcal{H}^2(G\times\widehat{G})$. As we mentioned above, the integral
$$\langle\varphi,\psi\rangle^{-1}_{L^2(G)}\int_{G\times\widehat{G}}{\rm tr}[\mathcal{G}_\psi f(y,\pi)M_\pi(L_y\varphi)]d\sigma(y,\pi),$$
is a well-defined function in $L^2(G)$, for convenience in calculations we use $f_{\psi}^\varphi$ for this function.
Using Corollary \ref{16}, for each $g\in L^2(G)$ we have
\begin{align*}
\langle f_{\psi}^\varphi,g\rangle_{L^2(G)}
&=\langle\varphi,\psi\rangle^{-1}_{L^2(G)}\int_{G\times\widehat{G}}{\rm tr}[\mathcal{G}_\psi f(y,\pi)\mathcal{G}_\varphi g(y,\pi)^*]d\sigma(y,\pi)
\\&=\langle\varphi,\psi\rangle^{-1}_{L^2(G)}\langle\mathcal{G}_\psi f,\mathcal{G}_\varphi g\rangle_{\mathcal{H}^2(G\times\widehat{G})}
\\&=\langle f,g\rangle_{L^2(G)}.
\end{align*}
Which follows $f=f_{\psi}^\varphi$ and so the inversion formula holds.
\end{proof}
As an immediate consequence of Theorem \ref{17} we have the following corollary.
\begin{corollary}
{\it Let $\psi$ be a window function such that $\|\psi\|_{L^2(G)}=1$. Then, for each $f\in L^2(G)$ we have
$$f=\int_{G\times\widehat{G}}{\rm tr}[\mathcal{G}_\psi f(y,\pi)M_\pi(L_y\psi)]d\sigma(y,\pi).$$
}\end{corollary}
Also when $G$ is compact the inversion theorem can be deduced more applicable in the following corollary.
\begin{corollary}
{\it Let $G$ be a compact group and $\psi,\varphi$ be two window functions such that $\langle\varphi,\psi\rangle_{L^2(G)}\not=0$. Then, for each $f\in L^2(G)$ we have
$$f(x)=\langle\varphi,\psi\rangle^{-1}_{L^2(G)}\int_{G}\sum_{[\pi]\in\widehat{G}}d_\pi{\rm tr}[\mathcal{G}_\psi f(y,\pi)M_\pi(L_y\varphi)(x)]dy
\hspace{0.25cm}{\rm for\ a.e}\hspace{0.25cm} x\in G.$$}
\end{corollary}
The following proposition gives us a useful relation of Gabor transform of two non-orthogonal window functions.
\begin{proposition}\label{2000}
{\it For window functions $\psi,\varphi$ with $\langle\varphi,\psi\rangle_{L^2(G)}\not=0$, we have $\mathcal{G}_\varphi^*\mathcal{G}_\psi=\langle\varphi,\psi\rangle_{L^2(G)} I_{L^2(G)}.$
}\end{proposition}
\begin{proof}
Let $S_\varphi:\mathcal{H}^2(G\times\widehat{G})\to L^2(G)$ be the bounded linear operator defined by
$$S_\varphi(K)=\int_{G\times\widehat{G}}{\rm tr}[K(y,\pi)M_\pi(L_y\varphi)]d\sigma(y,\pi).$$
Then $S_\varphi$ is the adjoint operator of $\mathcal{G}_\varphi$. In fact, using Proposition \ref{14}, for each $f\in L^2(G)$ and $K\in \mathcal{H}^2(G\times\widehat{G})$ we have
\begin{align*}
\langle S_\varphi(K),f\rangle_{L^2(G)}
&=\int_{G\times\widehat{G}}{\rm tr}[K(y,\pi)\mathcal{G}_\varphi f(y,\pi)^*]d\sigma(y,\pi)
\\&=\langle K,\mathcal{G}_\varphi f\rangle_{\mathcal{H}^2(G\times\widehat{G})}=\langle\mathcal{G}_\varphi^*(K),f\rangle_{L^2(G)}.
\end{align*}
Now, Theorem \ref{17} shows that $\mathcal{G}_\varphi^*\mathcal{G}_\psi=\langle\varphi,\psi\rangle_{L^2(G)} I_{L^2(G)}.$
\end{proof}

Let $\phi,\phi'\in L^2(G,\mathcal{B}(\mathcal{H}_\pi))$ and also let the operator
$\phi\otimes_\pi\phi':L^2(G)\to L^2(G,\mathcal{B}(\mathcal{H}_\pi))$ defined by
$$f\mapsto\phi\otimes_\pi\phi'(f)=\langle f,\phi'\rangle_\pi\phi.$$
Now Proposition (\ref{2000}) can be considered as the following continuous resolution of the identity operator for compact groups.
\begin{corollary}
{\it Let $G$ be a compact group and $\psi,\varphi$ window functions with $\langle\varphi,\psi\rangle_{L^2(G)}\not=0$. Then, we have
$$I_{L^2(G)}=\langle\varphi,\psi\rangle^{-1}_{L^2(G)}\int_{G}\sum_{[\pi]\in\widehat{G}}d_\pi{\rm tr}[M_\pi(L_y\varphi)\otimes_\pi M_\pi(L_y\psi)]dy,$$
where the right integral is a notation for the bounded linear operator defined on $L^2(G)$ by
$$f\mapsto \int_{G}\sum_{[\pi]\in\widehat{G}}d_\pi{\rm tr}[M_\pi(L_y\varphi)\otimes_\pi M_\pi(L_y\psi)(f)]dy.$$
}\end{corollary}


\section{\bf Examples}

As the first example we study the theory on the Heisenberg group. In quantum mechanics, Weyl-quantization (phase-space quantization) is a method for systematically associating a quantum mechanical Hermitian operator with a classical kernel function in phase space (see \cite{Weyl, Zach}). The Heisenberg group which is associated to n-dimensional quantum
mechanical systems plays an important role in the Weyl-quantization.

\subsection{Heisenberg group}

We start with the definitions and some description of the Heisenberg groups. The Heisenberg group $\mathbb{H}^n$ is a Lie group with the underlying manifold $\mathbb{R}^{d}$, where $d=2n+1$.
We denote points in $\mathbb{H}^n$ by ${\bf h}=(t,{\bf q},{\bf p})$ with $t\in\mathbb{R}$, ${\bf q},{\bf p}\in\mathbb{R}^n$ and define the group operation by
$$(t_1,{\bf q}_1,{\bf p}_1)(t_2,{\bf q}_2,{\bf p}_2)=
(t_1+t_2+\frac{1}{2}({\bf p}_1.{\bf q}_2-{\bf p}_2.{\bf q}_1),
{\bf q}_1+{\bf q}_2,{\bf p}_1+{\bf p}_2).$$
It is easy to justify that this is a group operation, with the identity element
$0=(0,\underline{0},\underline{0})$ and also the inverse of $(t,{\bf q},{\bf p})$ is given by
$(-t,-{\bf q},-{\bf p})$. It can be checked that Lebesgue measure on $\mathbb{R}^{2n+1}=\mathbb{H}^n$ is left and right invariant under the group action defined as above.
Hence, Lebesgue measure on $\mathbb{R}^{2n+1}$ gives the Haar measure on $\mathbb{H}^n$, and this group is unimodular. Taylor in \cite{Ta} proved that the following map
$$\varrho(t,{\bf q},{\bf p})f({\bf x})
=e^{i(t+{\bf q}.{\bf x}+{\bf q}.{\bf p}/2)}f({\bf x}+{\bf p}),\ {\rm for\ each}\ f\in L^2(\mathbb{R}^n),$$
is an irreducible unitary representation of $\mathbb{H}^n$ on the Hilbert space $L^2(\mathbb{R}^n)$. Now, for each $\lambda\not=0$, the map
$\delta_\lambda(t,{\bf q},{\bf p})=(\lambda t,{\rm sign}\lambda|\lambda|^{1/2}{\bf q},|\lambda|^{1/2}{\bf p}),$
is an automorphism of $\mathbb{H}^n$ and so $\pi_\lambda({\bf h}):=\varrho(\delta_\lambda({\bf h}))$ defines an
irreducible representation of $\mathbb{H}^n$ on the Hilbert space $L^2(\mathbb{R}^n)$. Each representation $\pi_\lambda$ is given explicitly on $L^2(\mathbb{R}^n)$ by
$$\pi_\lambda(t,{\bf q},{\bf p})f({\bf x})=e^{i\left(t\lambda+|\lambda|^{1/2}{\rm sign}\lambda{\bf q}.{\bf x}+\lambda{\bf q}.{\bf p}/2\right)}f({\bf x}+|\lambda|^{1/2}{\bf p}).$$
Except for the infinite-dimensional irreducible representations above, there are also the following one-dimensional irreducible representations of $\mathbb{H}^n$; $\pi_{(\underline{\xi},\underline{\eta})}(t,{\bf q},{\bf p})=e^{i(\underline{\xi}.{\bf q}+\underline{\eta}.{\bf p})}$.
Stone-von Neumann theorem and Kirillov theory imply that these two classes of the representations of $\mathbb{H}^n$, exhaust the irreducible representations of $\mathbb{H}^n$[see \cite{CG}]. Also, no two different representations of these two classes are unitarily equivalent.
Hence, we can say that
$$\widehat{\mathbb{H}^n}=\{\pi_\lambda|\lambda\in\mathbb{R}\backslash\{0\}\}
\bigcup\{\pi_{(\underline{\xi},\underline{\eta})}|(\underline{\xi},\underline{\eta})\in\mathbb{R}^n\times\mathbb{R}^n\}.$$
We associate to a function $f$ in $L^1(\mathbb{H}^n)\cap L^2(\mathbb{H}^n)$, the Fourier transform
$$\widehat{f}(\pi_{(\underline{\xi},\underline{\eta})})
=\int_{\mathbb{H}^n}f(t,{\bf q},{\bf p})e^{-i(\underline{\xi}.{\bf q}+
\underline{\eta}.{\bf p})}d(t,{\bf q},{\bf p}),\hspace{1cm}
\widehat{f}(\pi_\lambda)=\int_{\mathbb{H}^n}f(t,{\bf q},{\bf p})\pi_\lambda(t,{\bf q},{\bf p})^*d(t,{\bf q},{\bf p}).$$
For convenience, we write $\widehat{f}(\lambda)$ instead of $\widehat{f}(\pi_\lambda)$ and also $\widehat{f}(\underline{\xi},\underline{\eta})$ instead of $\widehat{f}(\pi_{(\underline{\xi},\underline{\eta})})$.
Then, the Plancherel theorem for the Heisenberg group $\mathbb{H}^n$ is given by
$$\int_{\mathbb{H}^n}|f(t,{\bf q},{\bf p})|^2d(t,{\bf q},{\bf p})=(2\pi)^{-(n+1)}\int_{-\infty}^{+\infty}\|\widehat{f}(\lambda)\|_{\rm HS}^2|\lambda|^nd\lambda.$$
Also, the following inversion Fourier transform for the Heisenberg group $\mathbb{H}^n$ holds;
$$f(t,{\bf q},{\bf p})=(2\pi)^{-(n+1)}\int_{-\infty}^{+\infty}{\rm tr}[\pi_\lambda(t,{\bf q},{\bf p})\widehat{f}(\lambda)]|\lambda|^nd\lambda.$$
Note that, the representations of the form $\pi_{(\underline{\xi},\underline{\eta})}$ have no contribution to the Plancherel formula and Fourier inversion transform because this set of representations has zero Plancherel measure. In other words, Plancherel measure $d\pi$ on $\widehat{\mathbb{H}^n}$ is given by $d\pi_\lambda=|\lambda|^nd\lambda$ and $d\pi_{(\underline{\xi},\underline{\eta})}=0$.
That is, the measure $(2\pi)^{-(n+1)}|\lambda|^n$ on $\mathbb{R}\backslash\{0\}$ is the Plancherel measure on $\widehat{\mathbb{H}^n}$.

\begin{example}
Let $\psi(t,{\bf q},{\bf p})=2^{d/4}e^{-\pi\|(t,{\bf q},{\bf p})\|^2}$ be the classical Gaussian window function on $\mathbb{H}^n$.

The associated continuous Gabor transform of $f\in\mathcal{C}_c(\mathbb{H}^n)$ with respect to the window function $\psi$ at each $(t,{\bf q},{\bf p},\pi_\lambda)\in \mathbb{H}^n\times\widehat{\mathbb{H}^n}$
can be identified by
\begin{align*}
\langle g,\mathcal{G}_\psi f(t,{\bf q},{\bf p},\pi_\lambda)k\rangle
&=\int_{\mathbb{H}^n}\mathcal{L}_{(t,{\bf q},{\bf p})}^\psi(f)(t',{\bf q'},{\bf p'})
\langle g,\pi_\lambda(t',{\bf q'},{\bf p'})^*k\rangle d(t',{\bf q}',{\bf p}')
\\&=\int_{\mathbb{H}^n}\int_{\mathbb{R}^n}\mathcal{L}_{(t,{\bf q},{\bf p})}^\psi(f)(t',{\bf q'},{\bf p'})
e^{i\left(t\lambda+|\lambda|^{1/2}{\rm sign}\lambda{\bf q}.{\bf{x}}+\lambda{\bf q}.{\bf p}/2\right)}
g({\bf{x}}+|\lambda|^{1/2}{\bf p})\overline{k({\bf x})}d{\bf x} d(t',{\bf q}',{\bf p}'),
\end{align*}
for each $g,k\in L^2(\mathbb{R}^n)$.

Similarly, the continuous Gabor transform at each $(t,{\bf q},{\bf p},\pi_{(\underline{\xi},\underline{\eta})})\in \mathbb{H}^n\times\widehat{\mathbb{H}^n}$ can be identified by
\begin{align*}
\mathcal{G}_\psi f(t,{\bf q},{\bf p},\pi_{(\underline{\xi},\underline{\eta})})
&=\int_{\mathbb{H}^n}\mathcal{L}_{(t,{\bf q},{\bf p})}^\psi(f)(t',{\bf q'},{\bf p'})
\pi_{(\underline{\xi},\underline{\eta})}(t',{\bf q'},{\bf p'})^* d(t',{\bf q}',{\bf p}')
\\&=\int_{\mathbb{H}^n}\mathcal{L}_{(t,{\bf q},{\bf p})}^\psi(f)(t',{\bf q'},{\bf p'})
e^{-i(\underline{\xi}.{\bf q}+\underline{\eta}.{\bf p})}d(t',{\bf q}',{\bf p}'),
\end{align*}
where $\mathcal{L}_{(t,{\bf q},{\bf p})}^\psi(f)(t',{\bf q'},{\bf p'})=2^{d/4}f(t',{\bf q}',{\bf p}')e^{-\pi(|t+t'+\frac{1}{2}({\bf p}.{\bf q'}-{\bf p'}.{\bf q})|^2+\|{\bf q}+{\bf q'}\|^2+\|{\bf p}+{\bf p'}\|^2)}.$
\end{example}

Using Theorem \ref{17}, for each $f\in L^2(\mathbb{H}^n)$ we can reconstruct $f$ via
\begin{align*}
\langle f,k\rangle_{L^2(G)}
&=\int_{\mathbb{H}^n}\int_{\widehat{\mathbb{H}^n}}{\rm tr}[\mathcal{F}\left(\mathcal{L}_{(t',{\bf q'},{\bf p'})}^\psi(f)\right)(\pi)\mathcal{G}_\psi k(t',{\bf q'},{\bf p'},\pi)]d\pi d(t',{\bf q}',{\bf p}')
\\&=(2\pi)^{-(n+1)}\int_{\mathbb{H}^n}\int_{\mathbb{R}\backslash\{0\}}{\rm tr}[\mathcal{F}\left(\mathcal{L}_{(t',{\bf q'},{\bf p'})}^\psi(f)\right)(\lambda)\mathcal{G}_\psi k(t',{\bf q'},{\bf p'},\lambda)]|\lambda|^nd\lambda d(t',{\bf q}',{\bf p}'),
\end{align*}
for each $k\in L^2(\mathbb{H}^n)$.
We recall that, for Lie group $\mathbb{R}^d$ and a window function $\psi$, the classical Gabor transform of $f$ in $L^2(\mathbb{R}^d)$ is
$$\displaystyle\mathcal{G}_\psi f({\bf x},{\bf w})=\int_{\mathbb{R}^d}e^{-2\pi i{\bf w}.{\bf y}}f({\bf y})\overline{\psi({\bf y}-{\bf x})}d{\bf y}.$$
one can consider the difference between these two transforms.

In the sequel, we provide another example of the continuous Gabor transform on the non-abelian matrix group $SL(2,\mathbb{R})$.
The special linear group $SL(2,\mathbb{R})$ has significant role in the theory of linear canonical transformation.
The linear canonical transformation is a generalization of the Fourier, fractional Fourier, Laplace, Gauss-Weierstrass, Bargmann and also the Fresnel transforms in particular cases (see \cite{Hen, Mosh}). Canonical transforms provide an appropriate tool for the analysis of a class of differential equations which are applicable in Fourier optics and quantum mechanics.

\subsection{Matrix group {SL}(2,$\mathbb{R}$)}

We recall that $G=SL(2,\mathbb{R})$ is the group of $2\times 2$ real matrices of determinant one. This group is unimodular and it's Haar integral which is right and left invariant is given by
$$\int_{-\infty}^{+\infty}\int_{-\infty}^{+\infty}\int_{-\infty}^{+\infty}\int_{-\infty}^{+\infty} f\left(
                      \begin{array}{cc}
                        x & y \\
                        z & t \\
                      \end{array}
                    \right)
dxdydzdt.$$
One can list continuous unitary representations of $SL(2,\mathbb{R})$ in five classes. For more explanations see \cite{FollH}.

(1) The trivial representation $\iota$, acting on $\mathbb{C}$.

(2) The discrete series $\{\delta_n^{\pm}:n\ge 2\}$. For $n\ge 2$, let $\mathcal{H}_n^+$ be the space of holomorphic functions $f$ on the upper half plane $U=\{z:\Im(z)>0\}$ such that
$$\|f\|_{[n]}^2=\int\int_U|f(x+iy)|^2y^{n-2}dxdy<\infty.$$
The representation $\delta_n^+$ of $SL(2,\mathbb{R})$ on $\mathcal{H}_n^+$ is defined by
$$\delta_n^+\left(
              \begin{array}{cc}
                a & b \\
                c & d \\
              \end{array}
            \right)
f(z)=(-bz+d)^{-n}f\left(\frac{az-c}{-bz+d}\right)\ {\rm for\ each}\ f\in \mathcal{H}_n^+.$$
Similarly, $\mathcal{H}_n^-$ is the space of anti-holomorphic functions $f$ on the upper half plane $U$ satisfying the condition
$\|f\|_{[n]}<\infty$ and the representation $\delta_n^{-}$ of $SL(2,\mathbb{R})$ on $\mathcal{H}_n^-$ is given the same as $\delta_n^+$.
It can be checked that, the representations $\{\delta_n^{\pm}:n\ge 2\}$ are unitary and irreducible.

(3) The mock discrete series $\{\delta_1^{\pm}\}$. Let $\mathcal{H}_1^+$ be the space of holomorphic functions $f$ on half plane $U$ such that
$$\|f\|_{[1]}^2=\sup_{y>0}\int_{-\infty}^{+\infty}|f(x+iy)|^2dx<\infty,$$
and also let $\mathcal{H}_1^-$ be the corresponding space of anti-holomorphic functions. The representations $\delta_1^\pm$ of ${SL(2,\mathbb{R})}$
on $\mathcal{H}_1^\pm$ are the same as $\delta_n^\pm$. The Hilbert spaces $\mathcal{H}_1^\pm$ can be naturally identified with certain subspaces of $L^2(\mathbb{R})$, namely
$\widetilde{\mathcal{H}_1^+}=\{f\in L^2(\mathbb{R}):\widehat{f}(\xi)=0\  {\rm for}\ \xi<0\}$, and
$\widetilde{\mathcal{H}_1^-}=\{f\in L^2(\mathbb{R}):\widehat{f}(\xi)=0\  {\rm for}\ \xi>0\}.$
The unitary map from $\mathcal{H}_1^\pm$ to $\widetilde{\mathcal{H}_1^\pm}$ simply takes a holomorphic or anti-holomorphic function on the upper half plane to its boundary values on $\mathbb{R}$ and the inverse map is given by the Fourier inversion formula, if $f$ is in $\widetilde{\mathcal{H}_1^+}$ or $\widetilde{\mathcal{H}_1^-}$, the corresponding $F\in{\mathcal{H}_1^+}$ or ${\mathcal{H}_1^-}$ is given by
$\displaystyle F(z)=\int e^{2\pi i\xi z}\widehat{f}(\xi)d\xi$ and $\displaystyle F(z)=\int e^{2\pi i\xi \overline{z}}\widehat{f}(\xi)d\xi,$
when these identifications are made, the representations $\delta_1^\pm$ are still given by discrete series but with $x\in\mathbb{R}$ replacing $z\in U$.

(4) The principal series $\{\pi_{it}^\pm:t\in\mathbb{R}\}$. These representations of $SL(2,\mathbb{R})$ induced from the one-dimensional representations of the upper triangular subgroup
$$P(2,\mathbb{R})=\left\{M_{a,b}=\left(
                                \begin{array}{cc}
                                  a & b \\
                                  0 & a^{-1} \\
                                \end{array}
                              \right)
:a\in\mathbb{R}\backslash\{0\}, b\in\mathbb{R}\right\}.$$
Moreover, any one-dimensional representation of $P(2,\mathbb{R})$ must annihilate its commutator subgroup, namely
$\{M_{a,b}:a=1\}$, so it is easily seen that these representations are precisely
$\beta_{it}^+(M_{a,b})=|a|^{it}$ and also $\beta_{it}^-(M_{a,b})=|a|^{it}{\rm sign}a$, for all $t\in\mathbb{R}$.
The principal series are then defined by
$\pi_{it}^+:={\rm ind}_{P(2,\mathbb{R})}^{SL(2,\mathbb{R})}(\xi_{it}^+),$ and also
$\pi_{it}^-:={\rm ind}_{P(2,\mathbb{R})}^{SL(2,\mathbb{R})}(\xi_{it}^-).$
For more details on Induced representations and also the notation ${\rm ind}_H^G(\pi)$, we refer the readers to \cite{FollH}.
The representations $\pi_{it}^-$ and $\pi_{it}^+$ are known as the spherical principal series and non-spherical principal series respectively. The Hilbert spaces for these representations, consist of complex valued functions on $SL(2,\mathbb{R})$ satisfying certain covariance condition on the cosets of $P(2,\mathbb{R})$ and such functions are determined by their values on  $$Q(2,\mathbb{R})=\left\{\left(
                                 \begin{array}{cc}
                                   1 & 0 \\
                                   t & 1 \\
                                 \end{array}
                               \right):t\in\mathbb{R}\right\}.$$
$Q(2,\mathbb{R})$ intersects each coset in exactly one point and since $Q(2,\mathbb{R})\cong\mathbb{R}$, the map $f\mapsto f|_{Q(2,\mathbb{R})}$ sets up a unitary isomorphism from these Hilbert spaces to $L^2(\mathbb{R})$, and that the resulting realization of the representation $\pi_{it}^\pm$ on $L^2(\mathbb{R})$ is given by
$$\pi_{it}^\pm\left(
                \begin{array}{cc}
                  a & b \\
                  c & d \\
                \end{array}
              \right)
f(x)=m_\pm(-bx+d)|-bx+d|^{-1-it}f\left(\frac{ax-c}{-bx+d}\right),$$
where $m_+(y)=1$ and $m_-(y)={\rm sign}(y)$. The principal series representations are all irreducible except for $\pi_0^-$ which is the direct sum of the mock discrete series $\delta_1^+$ and $\delta_1^-$, when all these representations are realized on subspaces of $L^2(\mathbb{R})$. Also, the representations $\pi_{-it}^+$ and $\pi_{-it}^-$ are equivalent respectively to $\pi_{it}^+$ and $\pi_{it}^-$ and otherwise these representations are all equivalent.

(5) The complementary series $\{\kappa_s:0<s<1\}$. The Hilbert space for $\kappa_s$ is the set of all complex valued functions $f$ on $\mathbb{R}$ such that
$$\|f\|_{(s)}^2=\frac{s}{2}\int_{-\infty}^{+\infty}\int_{-\infty}^{+\infty} f(x)\overline{f(y)}|x-y|^{s-1}dxdy<\infty,$$
and the action of $SL(2,\mathbb{R})$ is like that of the spherical principal series
$$\kappa_s\left(
            \begin{array}{cc}
              a & b \\
              c & d \\
            \end{array}
          \right)
f(x)=|-bx+d|^{-1-s}f\left(\frac{ax-c}{-bx+d}\right).$$
All these representation which had been mentions above are irreducible and equivalent except that $\pi_{-it}^\pm\cong\pi_{it}^\pm$ and $\pi_0^-\cong\delta_1^+\bigoplus\delta_1^-$ as we discussed about them and every irreducible representation of $SL(2,\mathbb{R})$ is equivalent to one of them.
Thus, we can identify the set of all irreducible representations of $\widehat{SL(2,\mathbb{R})}$ as follows;
$$\widehat{SL(2,\mathbb{R})}=\{\iota\}\bigcup\{\delta_n^{\pm}:n\ge 1\}\bigcup\{\pi_{it}^+:t\ge0\}\bigcup\{\pi_{it}^-:t>0\}\bigcup\{\kappa_s:0<s<1\}.$$

The Plancherel measure of the complementary and mock discrete series and trivial representation is zero, and on the principal and discrete series it is given by
$d\pi_{it}^+={t}/{2}\tanh({\pi t}/{2})dt$, $d\pi_{it}^-={t}/{2}\coth({\pi t}/{2})dt$ and also $d\{\delta_n^+\}=d\{\delta_n^-\}=n-1$.


\begin{example}\label{EX3}
Let $\psi(X)=2e^{-\pi{\|X\|}^2}$ be the Gaussian function on $SL(2,\mathbb{R})$, where $\|X\|$ is the Hilbert-Schmidt norm
 \footnote{ This norm is eventually called as the Frobenius norm or Schatten 2-norm.}of a matrix $X$. We compute the Gabor transform
 of $f\in \mathcal{C}_c(SL(2,\mathbb{R}))$ with respect to $\psi$ on $SL(2,\mathbb{R})\times\widehat{SL(2,\mathbb{R})}$ by

(1) For each $n\ge 2$, $h,k\in\mathcal{H}_n^\pm$ and $(X,\delta_n^\pm)\in SL(2,\mathbb{R})\times\widehat{SL(2,\mathbb{R})}$ we have
\begin{align*}
\langle k,\mathcal{G}_\psi f(X,\delta_n^\pm)h\rangle&=\int_{SL(2,\mathbb{R})}\mathcal{L}_X^\psi f(Y)\langle k,\delta_n^\pm(Y)^*h\rangle_{\mathcal{H}_n^\pm} dY
\\&=\int_{SL(2,\mathbb{R})}\int\int_U\mathcal{L}_X^\psi f(Y)\delta_n^\pm(Y)k(w)\overline{h(w)}\Im(w)^{n-2}dwdY
\\&=\int_{\mathbb{R}^4}\int\int_U\mathcal{L}_X^\psi f(Y){(t-yw)^{-n}k\left(\frac{xw-z}{t-yw}\right)}\overline{h(w)}\Im(w)^{n-2}dwd(x,y,z,t),
\end{align*}
where $\mathcal{L}_X^\psi f(Y)=2e^{-\pi\|X^{-1}Y\|^2}f(Y)$ and $Y=(x,y,z,t)$.

(2) For each $t\in\mathbb{R}$, $g,k\in L^2(\mathbb{R})$ and $(X,\pi_{it}^\pm)\in SL(2,\mathbb{R})\times\widehat{SL(2,\mathbb{R})}$ we have
\begin{align*}
\langle g,\mathcal{G}_\psi f(X,\pi_{it}^\pm)k\rangle_{L^2(\mathbb{R})}&=\int_{SL(2,\mathbb{R})}\mathcal{L}_X^\psi f(Y)\langle g,\pi_{it}^\pm(Y)^*k\rangle_{L^2(\mathbb{R})} dY
\\&=\int_{SL(2,\mathbb{R})}\int_\mathbb{R}\mathcal{L}_X^\psi f(Y)\pi_{it}^\pm(Y)k(v)\overline{g(v)}dvdY
\\&=\int_{\mathbb{R}^4}\int_\mathbb{R}\mathcal{L}_X^\psi f(Y)m_\pm(-yv+u)|-yv+u|^{-1-it}k\left(\frac{xv-z}{-yv+u}\right)\overline{g(v)}dvd(x,y,z,u),
\end{align*}
where $\mathcal{L}_X^\psi f(Y)=2e^{-\pi\|X^{-1}Y\|^2}f(Y)$ and $Y=(x,y,z,u)$.

(3) For each $0<s<1$, $h,k\in\mathcal{H}_{(s)}$ and $(X,\kappa_s)\in SL(2,\mathbb{R})\times\widehat{SL(2,\mathbb{R})}$ we have
\begin{align*}
\langle h,\mathcal{G}_\psi f(X,\kappa_s)k\rangle_{(s)}&=\int_{SL(2,\mathbb{R})}\mathcal{L}_X^\psi f(Y)\langle h,\kappa_s(Y)^*k\rangle_{(s)} dY
\\&=\frac{s}{2}\int_{SL(2,\mathbb{R})}\int_{-\infty}^{+\infty}\int_{-\infty}^{+\infty}\mathcal{L}_X^\psi f(Y)\kappa_s(Y)k(v)\overline{h(u)}|u-v|^{1-s}dvdudY
\\&=\frac{s}{2}\int_{\mathbb{R}^4}\int_{-\infty}^{+\infty}\int_{-\infty}^{+\infty}\mathcal{L}_X^\psi f(Y)k\left(\frac{xv-z}{-yv+t}\right)\overline{h(u)}|u-v|^{1-s}|-yv+t|^{-1-s}dvdud(x,y,z,t),
\end{align*}
where $\mathcal{L}_X^\psi f(Y)=2e^{-\pi\|X^{-1}Y\|^2}f(Y)$ and $Y=(x,y,z,t)$.
\end{example}
Using Theorem \ref{17}, we can reconstruct each $f\in L^2(SL(2,\mathbb{R}))$ via
\begin{align*}
\langle f,k\rangle_{L^2(G)}&
=\int_{\mathbb{R}^4}\int_{\widehat{SL(2,\mathbb{R})}}{\rm tr}[\widehat{\mathcal{L}_X^\psi f}(\pi)\mathcal{G}_\psi k(Y,\pi)]d\pi dY
\\&=\int_{\mathbb{R}^4}\int_{\widehat{SL(2,\mathbb{R})}}{\rm tr}[\widehat{\mathcal{L}_X^\psi f}(\pi_{it}^+)\mathcal{G}_\psi k(y,\pi_{it}^+)]d\pi_{it}^+dY+
\int_{\mathbb{R}^4}\int_{\widehat{SL(2,\mathbb{R})}}{\rm tr}[\widehat{\mathcal{L}_X^\psi f}(\pi_{it}^-)\mathcal{G}_\psi k(Y,\pi_{it}^-)]d\pi_{it}^-dY
\\&\ \ +\int_{\mathbb{R}^4}\int_{\widehat{SL(2,\mathbb{R})}}{\rm tr}[\widehat{\mathcal{L}_X^\psi f}(\delta_n^+)\mathcal{G}_\psi k(Y,\delta_n^+)]d\delta_n^+dY
+\int_{\mathbb{R}^4}\int_{\widehat{SL(2,\mathbb{R})}}{\rm tr}[\widehat{\mathcal{L}_X^\psi f}(\delta_n^-)\mathcal{G}_\psi k(Y,\delta_n^-)]d\delta_n^-dY
\\&=\frac{1}{2}\int_{\mathbb{R}^4}\int_0^\infty\left(t\tanh \frac{\pi t}{2}{\rm tr}[\widehat{\mathcal{L}_X^\psi f}(\pi_{it}^+)\mathcal{G}_\psi k(Y,\pi_{it}^+)]+t\coth \frac{\pi t}{2}{\rm tr}[\widehat{\mathcal{L}_X^\psi f}(\pi_{it}^-)\mathcal{G}_\psi k(Y,\pi_{it}^-)]\right)
dtdY\\&\ \ +\int_{\mathbb{R}^4}\left(\sum_{n=1}^\infty(n-1)\left({\rm tr}[\widehat{\mathcal{L}_X^\psi f}(\delta_n^+)\mathcal{G}_\psi k(Y,\delta_n^+)]+{\rm tr}[\widehat{\mathcal{L}_X^\psi f}(\delta_n^-)\mathcal{G}_\psi k(Y,\delta_n^-)]\right)\right)dY,
\end{align*}
for each $k\in L^2(SL(2,\mathbb{R}))$.


\bibliographystyle{amsplain}

\end{document}